\documentclass[11pt]{amsart}
\usepackage[all]{xy}
\usepackage[all]{xy} \usepackage{amssymb}
\usepackage{wasysym}
\usepackage{amsmath} 
\usepackage{graphicx} 
\usepackage[english]{babel} 
\usepackage{epsfig}
\usepackage{epstopdf}
\usepackage{enumerate}
\usepackage{color}
\usepackage[color=blue!50]{todonotes}

\swapnumbers

\theoremstyle{definition}

\newtheorem{remark}[subsection]{Remark}
\theoremstyle{plain}
\newtheorem{theorem}[subsection]{Theorem}
\newtheorem{theoreme}[subsection]{Th\'eor\`eme}

\newtheorem{lemma}[subsection]{Lemma}

\def\arbreA{\def\svgscale{0.02}
\begingroup%
  \makeatletter%
  \providecommand\color[2][]{%
    \errmessage{(Inkscape) Color is used for the text in Inkscape, but the package 'color.sty' is not loaded}%
    \renewcommand\color[2][]{}%
  }%
  \providecommand\transparent[1]{%
    \errmessage{(Inkscape) Transparency is used (non-zero) for the text in Inkscape, but the package 'transparent.sty' is not loaded}%
    \renewcommand\transparent[1]{}%
  }%
  \providecommand\rotatebox[2]{#2}%
  \ifx\svgwidth\undefined%
    \setlength{\unitlength}{493.40162673bp}%
    \ifx\svgscale\undefined%
      \relax%
    \else%
      \setlength{\unitlength}{\unitlength * \real{\svgscale}}%
    \fi%
  \else%
    \setlength{\unitlength}{\svgwidth}%
  \fi%
  \global\let\svgwidth\undefined%
  \global\let\svgscale\undefined%
  \makeatother%
  \begin{picture}(1,0.34069187)%
    \put(0,0){\includegraphics[width=\unitlength]{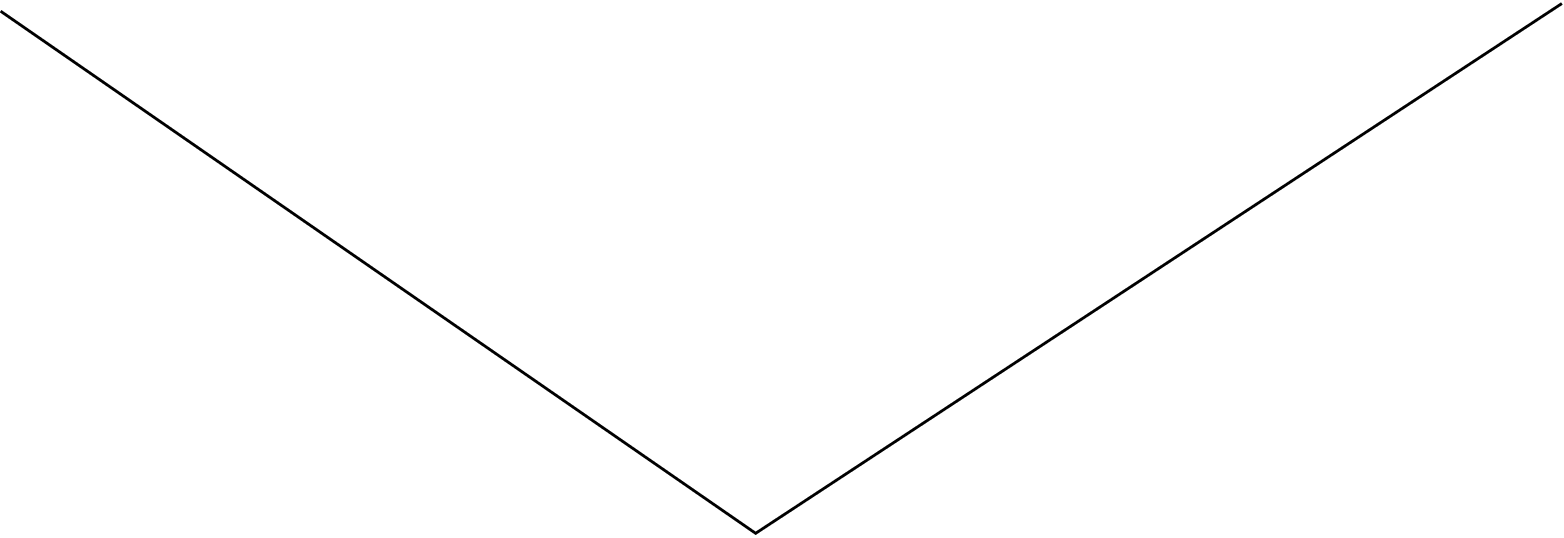}}%
  \end{picture}%
\endgroup%
}

\def\arbreBA{\def\svgscale{0.03}
\begingroup%
  \makeatletter%
  \providecommand\color[2][]{%
    \errmessage{(Inkscape) Color is used for the text in Inkscape, but the package 'color.sty' is not loaded}%
    \renewcommand\color[2][]{}%
  }%
  \providecommand\transparent[1]{%
    \errmessage{(Inkscape) Transparency is used (non-zero) for the text in Inkscape, but the package 'transparent.sty' is not loaded}%
    \renewcommand\transparent[1]{}%
  }%
  \providecommand\rotatebox[2]{#2}%
  \ifx\svgwidth\undefined%
    \setlength{\unitlength}{544.30614828bp}%
    \ifx\svgscale\undefined%
      \relax%
    \else%
      \setlength{\unitlength}{\unitlength * \real{\svgscale}}%
    \fi%
  \else%
    \setlength{\unitlength}{\svgwidth}%
  \fi%
  \global\let\svgwidth\undefined%
  \global\let\svgscale\undefined%
  \makeatother%
  \begin{picture}(1,0.3310869)%
    \put(0,0){\includegraphics[width=\unitlength]{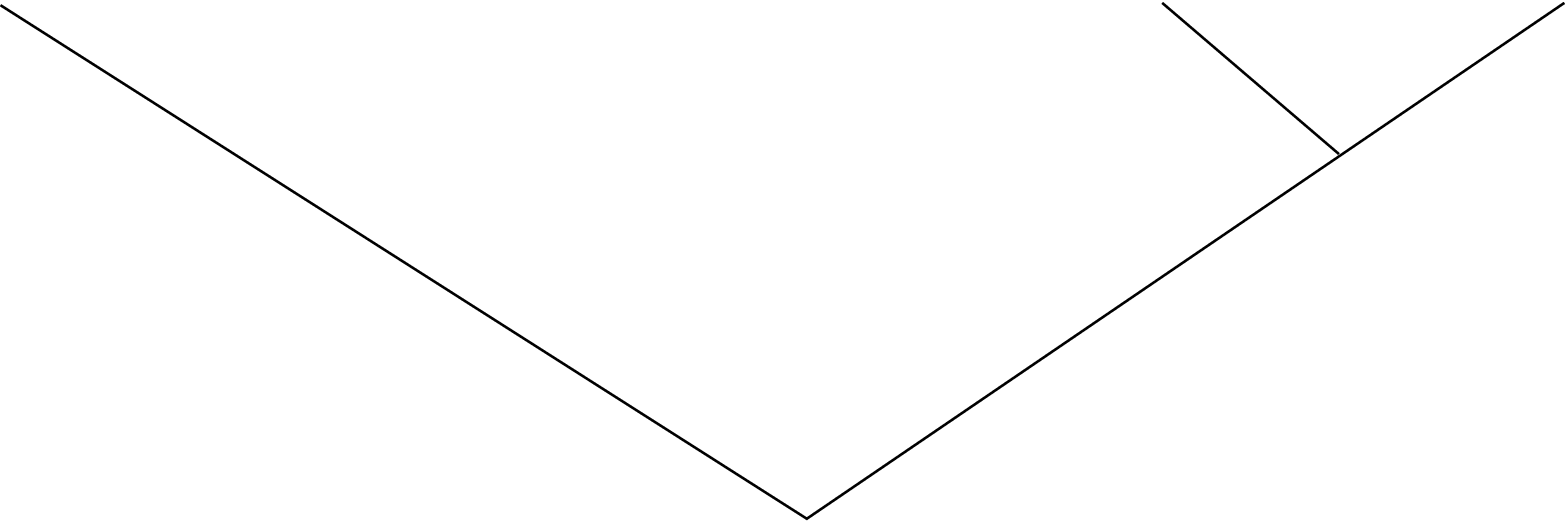}}%
  \end{picture}%
\endgroup%
} 

\def\arbreCBA{\def\svgscale{0.04}
\begingroup%
  \makeatletter%
  \providecommand\color[2][]{%
    \errmessage{(Inkscape) Color is used for the text in Inkscape, but the package 'color.sty' is not loaded}%
    \renewcommand\color[2][]{}%
  }%
  \providecommand\transparent[1]{%
    \errmessage{(Inkscape) Transparency is used (non-zero) for the text in Inkscape, but the package 'transparent.sty' is not loaded}%
    \renewcommand\transparent[1]{}%
  }%
  \providecommand\rotatebox[2]{#2}%
  \ifx\svgwidth\undefined%
    \setlength{\unitlength}{544.30614828bp}%
    \ifx\svgscale\undefined%
      \relax%
    \else%
      \setlength{\unitlength}{\unitlength * \real{\svgscale}}%
    \fi%
  \else%
    \setlength{\unitlength}{\svgwidth}%
  \fi%
  \global\let\svgwidth\undefined%
  \global\let\svgscale\undefined%
  \makeatother%
  \begin{picture}(1,0.3310869)%
    \put(0,0){\includegraphics[width=\unitlength]{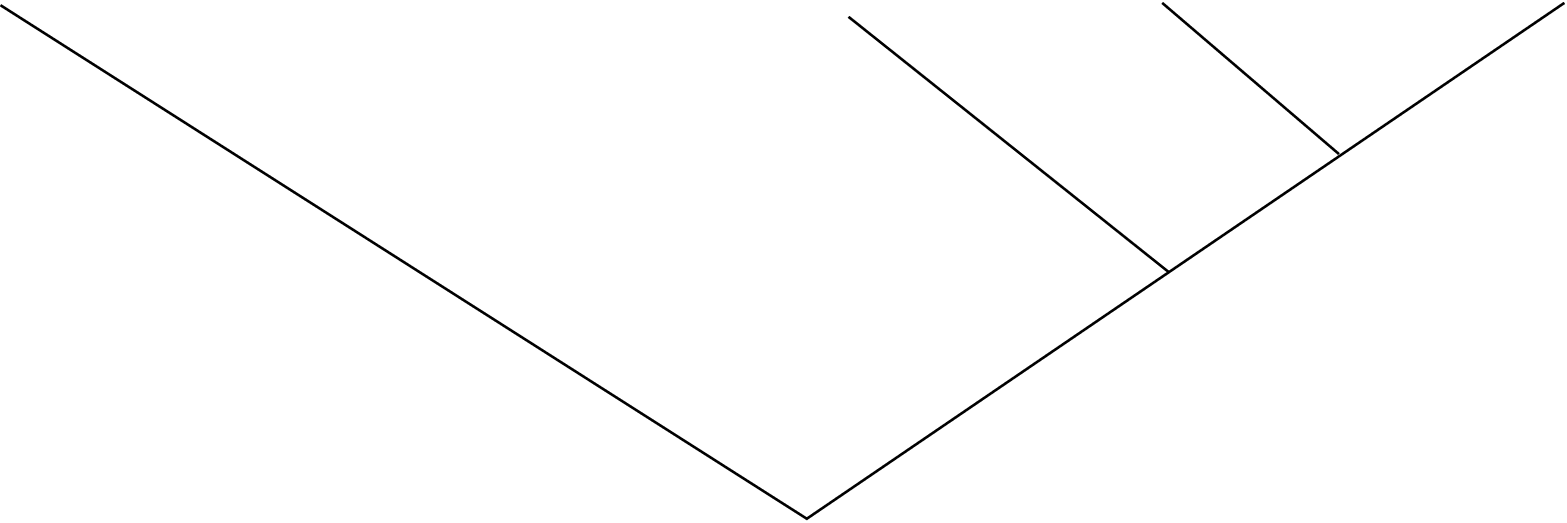}}%
  \end{picture}%
\endgroup%
}

\def\arbreDCBA{\def\svgscale{0.04}
\begingroup%
  \makeatletter%
  \providecommand\color[2][]{%
    \errmessage{(Inkscape) Color is used for the text in Inkscape, but the package 'color.sty' is not loaded}%
    \renewcommand\color[2][]{}%
  }%
  \providecommand\transparent[1]{%
    \errmessage{(Inkscape) Transparency is used (non-zero) for the text in Inkscape, but the package 'transparent.sty' is not loaded}%
    \renewcommand\transparent[1]{}%
  }%
  \providecommand\rotatebox[2]{#2}%
  \ifx\svgwidth\undefined%
    \setlength{\unitlength}{544.30614828bp}%
    \ifx\svgscale\undefined%
      \relax%
    \else%
      \setlength{\unitlength}{\unitlength * \real{\svgscale}}%
    \fi%
  \else%
    \setlength{\unitlength}{\svgwidth}%
  \fi%
  \global\let\svgwidth\undefined%
  \global\let\svgscale\undefined%
  \makeatother%
  \begin{picture}(1,0.3310869)%
    \put(0,0){\includegraphics[width=\unitlength]{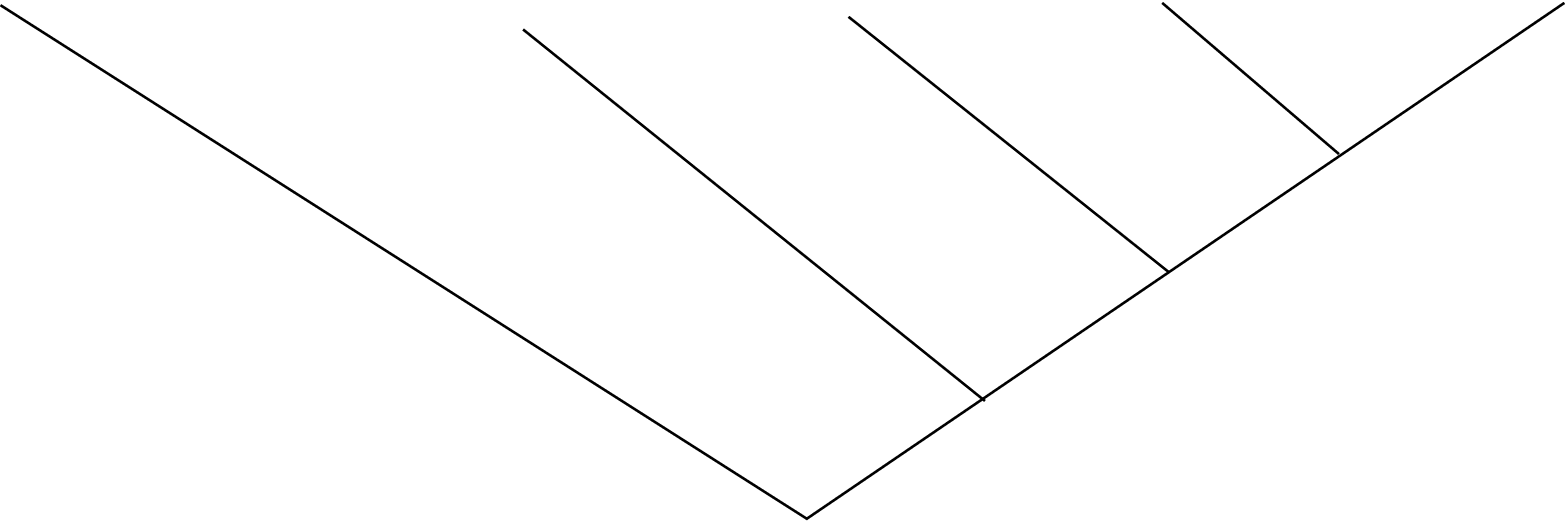}}%
  \end{picture}%
\endgroup%
}

\def\arbreADBA{\def\svgscale{0.04}
\begingroup%
  \makeatletter%
  \providecommand\color[2][]{%
    \errmessage{(Inkscape) Color is used for the text in Inkscape, but the package 'color.sty' is not loaded}%
    \renewcommand\color[2][]{}%
  }%
  \providecommand\transparent[1]{%
    \errmessage{(Inkscape) Transparency is used (non-zero) for the text in Inkscape, but the package 'transparent.sty' is not loaded}%
    \renewcommand\transparent[1]{}%
  }%
  \providecommand\rotatebox[2]{#2}%
  \ifx\svgwidth\undefined%
    \setlength{\unitlength}{544.30614828bp}%
    \ifx\svgscale\undefined%
      \relax%
    \else%
      \setlength{\unitlength}{\unitlength * \real{\svgscale}}%
    \fi%
  \else%
    \setlength{\unitlength}{\svgwidth}%
  \fi%
  \global\let\svgwidth\undefined%
  \global\let\svgscale\undefined%
  \makeatother%
  \begin{picture}(1,0.3310869)%
    \put(0,0){\includegraphics[width=\unitlength]{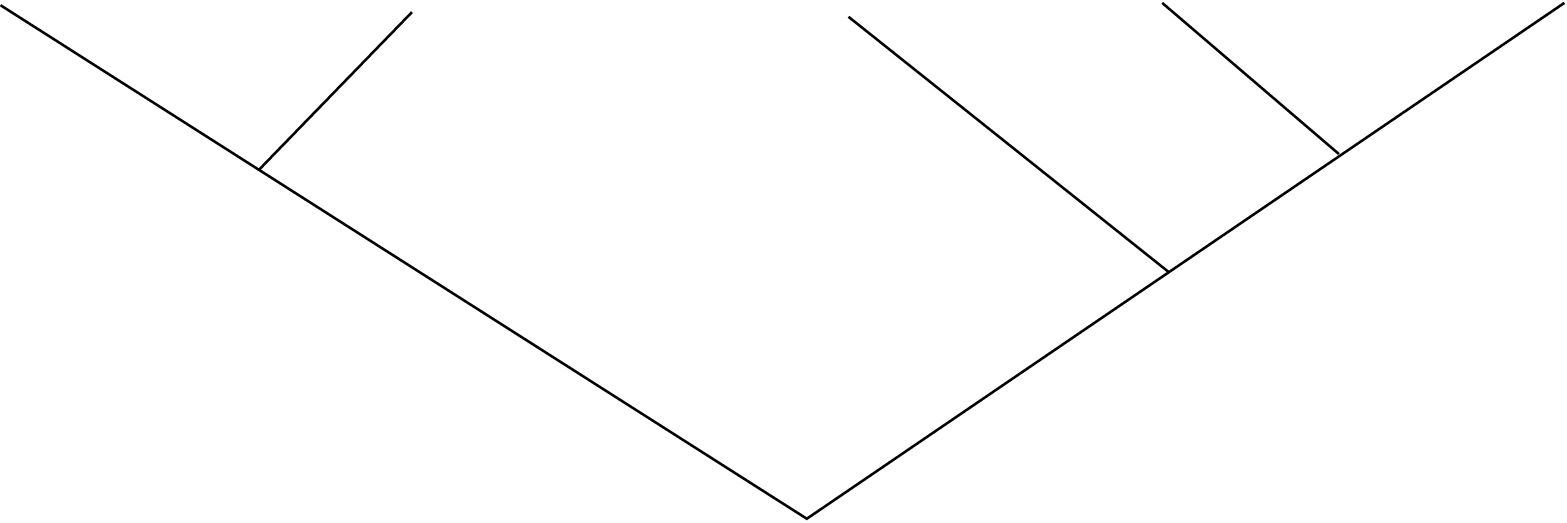}}%
  \end{picture}%
\endgroup%
}

\def\PreLieNAP {\vcenter{\xymatrix@R=1pt@C=1pt{
Pre-Lie\circ Pre-Lie\ar@{->}^{\gamma_{PL}}[ddd]  \ar@{->}^{LPP}[dddrr] &\cong_{_{\mathbb S-mod}}&NAP\circ NAP\ar@{->}^{\gamma_{NAP}}[ddd] &&& \\
  &&&& \\
&&&&\\
Pre-Lie&&NAP&&}}}

\def\ignore#1{}

 \newcommand{\commentc}[1]{}

\begin{document}

\author[E. Burgunder, B. Delcroix-Oger, D.Manchon]{Emily Burgunder, B\'er\'enice Delcroix--Oger, Dominique Manchon}
\address{EB, BDO: Universit\'e Paul Sabatier\\
Institut de Math\'ematiques de Toulouse\\
118 route de Narbonne\\
F-31062 Toulouse Cedex 9 France}
\email{burgunder@math.univ-toulouse.fr}
\email{berenice.delcroix@math.univ-toulouse.fr}
\address{DM: Laboratoire de Math\'ematiques Blaise Pascal\\
Universit\'e Clermont-Auvergne\\
3 place Vasar\'ely\\
CS 60026\\
F63178 Aubi\`ere, France}
\email{manchon@math.univ-bpclermont.fr}

\title{An operad is never free as a pre-Lie algebra}
\keywords{operads, pre-Lie, trees, forest}
\thanks{Our joint work was partially supported by ANR CARMA while in the Plantiers.}


\begin{abstract} 
An operad is naturally endowed with a pre-Lie structure. We prove that as  a pre-Lie algebra an operad is not free. The proof holds on defining a non-vanishing linear operation in the pre-Lie algebra which is zero in any operad. 

Une op\'erade admet une structure d'alg\`ebre pr\'e-Lie naturelle. Nous montrons qu'une op\'erade n'est jamais libre en tant qu'alg\`ebre pr\'e-Lie. La preuve repose sur une relation pr\'e-Lie non nulle qui s'annule dans le cadre op\'eradique.  
\end{abstract}

\maketitle

\section{Version abr\'eg\'ee en fran\c{c}ais}
Soit $(V,\triangleleft)$ une alg\`ebre pr\'e-Lie, c'est-\`a-dire un espace vectoriel muni d'une op\'eration $\triangleleft:V\otimes V\to V$ v\'erifiant la relation
$$( x\triangleleft y)\triangleleft z- x\triangleleft (y\triangleleft z)  =( x\triangleleft z)\triangleleft y-  x\triangleleft (z\triangleleft y)$$
pour tout $x,y,z\in V$. D\'efinissons par r\'ecurrence les \textsl{\'el\'ements d'insertion} suivants :
\begin{align*}
 t\triangleleft{(s_1,s_2)}:=& (t\triangleleft s_1)\triangleleft s_2-t\triangleleft (s_1\triangleleft s_2)\ , \\
 t\triangleleft{(s_1,s_2,s_3)}:=& \big(t\triangleleft{(s_1,s_2)}\big)\triangleleft s_3-t\triangleleft{(s_1\triangleleft s_3,s_2)}-t\triangleleft{(s_1,s_2\triangleleft s_3)}\ , \\
t\triangleleft{(s_1,s_2,\ldots,s_n)}:=& \big( t\triangleleft{(s_1,s_2,\ldots,s_{n-1})}\big)\triangleleft s_n-\sum_{i=1}^{n-1}t\triangleleft{(s_1,\ldots,s_i\triangleleft s_n,\ldots,s_{n-1})},
\end{align*}
avec $t,s_1,\ldots,s_n\in V$. L'\'el\'ement  $t\triangleleft{(s_1,s_2,\ldots,s_n)}$ est invariant par permutation des $s_i,i=1,\ldots,n$, et non nul si $V$ est l'alg\`ebre pr\'e-Lie libre. En effet, en termes d'arbres enracin\'es cf. \cite{Chapoton-Livernet}, cet \'el\'ement se comprend comme le branchement de $s_1,\ldots, s_n$ sur $t$ sans branchement de $s_i$ sur $s_j$ pour $i\neq j\in \{1,\ldots, n\}$.\\

Une op\'erade $\mathcal P$, plus pr\'ecis\'ement $ \mathcal P:=\bigoplus_{n\ge 1}\mathcal P(n)$, est naturellement munie d'une structure d'alg\`ebre pr\'e-Lie d\'efinie en utilisant les compositions partielles $\circ_i:\mathcal P(n)\otimes \mathcal P(m)\to\mathcal P(n+m-1):\mu\otimes \nu\mapsto \mu\circ_i\nu:= \mu(\mbox{id},\cdots, \mbox{id},\underbrace{\nu}_{i^{\textrm{\`eme}} \textrm{ entr\'ee }},\mbox{id},\cdots \mbox{id})$ comme suit~: 
$\forall \mu\in \mathcal P(n),\nu\in\mathcal P(m), \ \mu\triangleleft \nu=\sum_{i=1}^n\mu\circ_i\nu$
 o\`u la somme est \'etendue lin\'eairement.

\begin{theoreme}
Une op\'erade non triviale n'est pas libre en tant qu'alg\`ebre pr\'e-Lie.
\end{theoreme}

En effet, dans une op\'erade, les \'el\'ements d'insertion $\mu\triangleleft(\nu_1,\ldots,\nu_{n+k})$ sont nuls d\`es que l'on consid\`ere une op\'eration $\mu$ d'arit\'e $n$, avec $k\in\mathbb N^*$. 
\section{Pre-Lie insertion laws in operads} 
\subsection{Pre-Lie insertion laws}
A pre-Lie algebra is a vector space $V$ together with a bilinear map $\triangleleft:V\otimes V\to V$ verifying the following relation:
$$\forall x,y,z \in V,\ ( x\triangleleft y)\triangleleft z- x\triangleleft (y\triangleleft z)  =( x\triangleleft z)\triangleleft y-  x\triangleleft (z\triangleleft y).$$

Let $t,s_1,s_2,\ldots, s_n\in V$. Define recursively the following \textsl{insertion elements}:
\begin{align*}
 t\triangleleft{(s_1,s_2)}:=& (t\triangleleft s_1)\triangleleft s_2-t\triangleleft (s_1\triangleleft s_2)\ , \\
 t\triangleleft{(s_1,s_2,s_3)}:=& \big(t\triangleleft{(s_1,s_2)}\big)\triangleleft s_3-t\triangleleft{(s_1\triangleleft s_3,s_2)}-t\triangleleft{(s_1,s_2\triangleleft s_3)}\ , \\
t\triangleleft{(s_1,s_2,\ldots,s_n)}:=& \big( t\triangleleft{(s_1,s_2,\ldots,s_{n-1})}\big)\triangleleft s_n-\sum_{i=1}^{n-1}t\triangleleft{(s_1,\ldots,s_i\triangleleft s_n,\ldots,s_{n-1})}.
\end{align*}
By definition of a pre-Lie algebra, these insertion elements are invariant under permutation of the $s_i$'s. They can be understood combinatorially in the free pre-Lie algebra when viewed as the vector space spanned by rooted trees with its grafting operation defined in Chapoton-Livernet \cite{Chapoton-Livernet}, as follows: $t\triangleleft{(s_1,s_2,\ldots,s_n)}$ is obtained by grafting the trees $s_i$ on $t$ such that no $s_i$ is grafted on an $s_j$. Therefore $t\triangleleft{(s_1,s_2,\ldots,s_n)}$ is nonzero. For example if $t=s_1=\ldots=s_n$ is the pre-Lie generating element, i.e., in terms of rooted trees, a single root, then the $n^{\textrm{th}}$ pre-Lie insertion element is the $n$-leaved corolla: see figure \ref{insertionlaw} for the grafting of $3$ roots on a root.

\begin{figure}
\begin{tikzpicture}
[level distance=0mm,edge from parent/.style={draw=none},
every node/.style={circle,inner sep=1pt, draw}]
\node{1}[grow'=up];
\end{tikzpicture}
$\ \triangleleft\ \Big($
\begin{tikzpicture}
[level distance=0mm,edge from parent/.style={draw=none},
every node/.style={circle,inner sep=1pt, draw}]
\node{3}[grow'=up]
child {node {2}}
child {}
child {node {4}};
\end{tikzpicture}
$\Big)=$
\begin{tikzpicture}
[level distance=10mm,
every node/.style={circle,inner sep=1pt, draw}]
\node  {1} [grow'=up]
child {node {2}}
child {node {3}}
child {node {4}};
\end{tikzpicture}
\caption{ \label{insertionlaw}}
\end{figure}
\subsection{Insertion pre-Lie law on operations of an operad}
An operad $\mathcal P$ is a collection of $\mathbb S_n$-modules $\mathcal P(n)$ together with a substitution, which is a $\mathbb S$-module morphism, $\mathcal P\circ\mathcal P\to \mathcal P$, and the unit map, a morphism of $\mathbb S$-modules $\eta:I\to\mathcal P$ satisfying the associativity and unitarity axioms (axioms for monoids) with $I=(0,\mathbb K\cdot \mbox{id},0,\ldots, 0)$. We will restrict to operads with $\mathcal P(0)=0$ and $\mathcal P(1)=\mathbb K\cdot \textrm{id}$. An equivalent definition, known as the partial definition of an operad,  is given by the partial composition of two operations defined by substitution : $\circ_i:\mathcal P(n)\otimes \mathcal P(m)\to\mathcal P(n+m-1):\mu\otimes \nu\mapsto \mu\circ_i\nu:= \mu(\mbox{id},\cdots, \mbox{id},\underbrace{\nu}_{i^{\textrm{th}} \textrm{ entry }},\mbox{id},\cdots \mbox{id})$: an operad is then a collection of $\mathbb S_n$-modules $\mathcal P(n)$ together with partial compositions compatible with the symmetric group action, and satisfying the  associativity of partial compositions depending on the relative positions of the two graftings (sequential composition \eqref{eq:sequential} and parallel composition \eqref{eq:parallel}) and compatibility with the identity operation \eqref{eq:unity}  that is to say:
\begin{align}
&(\lambda\circ_i\mu)\circ_{i-1+j}\nu=\lambda\circ_i(\mu\circ_j\nu) & \forall 1\leq i\leq l, 1\leq j\leq \ m,\label{eq:sequential}\\
&(\lambda\circ_i\mu)\circ_{k-1+m} \nu=(\lambda\circ_k\nu)\circ_i\mu&\forall 1\leq i \leq k\leq l\label{eq:parallel}\\
&{\mbox{id}}\circ_1\nu=\nu\ ,\ \mu\circ_i{\mbox{id}}=\mu\label{eq:unity}
\end{align}
for any $\lambda\in\mathcal P(l),\ \mu\in\mathcal P(m),\ \nu\in\mathcal P(n)$ and the identity element ${\mbox{id}}\in\mathcal P(1)$. For more details see for example \cite{Loday-Vallette}.\\

An operad $\mathcal P$ can be endowed with a natural graded pre-Lie algebra structure, the grading being given by arity minus one. There are several versions of this construction leading to four different pre-Lie algebras:
\begin{eqnarray*}
\mathcal P&:=&\bigoplus_{n\ge 1}\mathcal P(n),\\
\mathcal P^+&:=&\bigoplus_{n\ge 2}\mathcal P(n),\\
\underline{\mathcal P}&:=&\bigoplus_{n\ge 1}\mathcal P(n)_{S_n},\\
\underline{\mathcal P}^+&:=&\bigoplus_{n\ge 2}\mathcal P(n)_{S_n},\\
\end{eqnarray*}
The first one is mentioned in \cite[Proposition 5.2.20]{Loday-Vallette}, the third one is due to Kapranov and Manin \cite{KM01} (see also \cite{Chapoton}), the second and fourth have been proposed in \cite{FM15}.  The pre-Lie product is defined as follows: 
$$\forall \mu\in \mathcal P(n),\nu\in\mathcal P(m), \ \mu\triangleleft \nu=\sum_{i=1}^n\mu\circ_i\nu.$$
The pre-Lie relation reads as such: the composition of two operations on different entries of a third operation does not depend on the order in which the  partial compositions are performed, by virtue of the parallel axioms of operads. The Lie algebra underlying $\mathcal P^+$ is pro-unipotent. This obviously defines a pre-Lie product on the coinvariants, leading to the third and fourth variants. The Lie algebra underlying $\underline{\mathcal P}^+$ is of course pro-unipotent as well. The element $\hbox{id}\in\mathcal P_1$ verifies:
$$\hbox{id}\triangleleft \mu=\mu,\hskip 12mm \mu\triangleleft\hbox{id}=n\mu$$
for any $\mu\in\mathcal P(n)$.\\

The free operad on some generating operations can be understood combinatorially as the vector space spanned by planar reduced rooted trees with internal nodes decorated with the generating operations, endowed with the following partial operations : for two trees $t,s$, the partial composition $t\circ_i s$ is the plugging of the tree $s$ into the $i^{\textrm{th}}$ leaf of $t$, or equivalently substitute the $i^{\textrm{th}}$ leaf of $t$ by $s$. In this setting the insertion element  $t\triangleleft{(s_1,s_2,\ldots,s_n)}$ for operations $t,s_1,\ldots, s_n$, viewed as trees, is defined as the plugging of the trees $s_i$'s on $t$, no $s_i$'s being plugged on to each other.  This property is the key element of the proof of the non-freeness of operads.
\section{Operads are not free pre-Lie algebras}
 In \cite{Manchon-Saidi} the third author and A. Sa{\"{\i}}di have proved that the pre-Lie operad is not free as a pre-Lie algebra. This is part of a more general phenomenon:
\begin{theorem}
Operads are not free as pre-Lie algebras.
\end{theorem}
\ignore{
\begin{remark} We can also consider unitary pre-Lie algebras, that is to say endowed with an element $1$ such that $1\triangleleft 1=1$ and $x\triangleleft 1=1\triangleleft x=x$. The natural pre-Lie structure on $\mathcal P_+$ naturally extends to a natural unitary pre-Lie algebra structure on $\mathcal P=\bigoplus_{n\ge 1}\mathcal P(n)$. Any operad $\mathcal P$ is not free as non-unitary pre-Lie algebra neither.  Indeed, suppose there exists  a map $\mathcal P \to PreLie(V)$ for a certain vector space $V$, which is an isomorphism of pre-Lie algebras. Consider the combinatorial description of the free pre-Lie algebra as in \cite{Chapoton-Livernet}. The operation $\mbox{id}$ will be sent on a root, as it is a generating operation. But then in $\mathcal P$, we get $\mbox{id}\triangleleft(\mbox{id},{\mbox{id}})=0$ which in the setting of unitary pre-Lie algebras  is non vanishing: indeed it amounts to the associator applied to the root  see figure \ref{nonunitary}.
\end{remark}
}
\begin{lemma}
Let $\mathcal P$ be an operad endowed with its pre-Lie algebra structure $\triangleleft$. For any  operation $\mu$  of arity $n$, and any $\nu_1,\ldots,\nu_{n+1}$ the $(n+1)^{\textrm{th}}$-insertion pre-Lie element vanishes : $\mu\triangleleft{(\nu_1,\nu_2,\ldots,\nu_{n+1})}=0$.
\end{lemma}
\begin{proof}Direct computation of the associator in the operad settings using the sequential and parallel composition axioms is easier to make explicit with the use of the operations $\circ_{i,j}:\mathcal P(l)\otimes \mathcal P(m_1)\otimes \mathcal P(m_2)\to \mathcal P(l+m_1+m_2-2)$ defined by:
$$\mu\circ_{i,j}(\nu_1,\nu_2):= \mu(\mbox{id},\ldots, \mbox{id},\underbrace{\nu_1}_{i^{\textrm{th}} \textrm{ entry }},\mbox{id},\ldots ,\mbox{id}, \underbrace{\nu_2}_{j^{\textrm{th}} \textrm{ entry }},\mbox{id},\ldots,\mbox{id}).$$
We then get:
\begin{align*}
\mu\triangleleft(\nu_1,\nu_2)=& \sum_{i=1}^{l}\sum_{\tiny{\begin{array}{c}j=1\\j\neq i\end{array}}}^{l} \mu\circ_{i,j}(\nu_1,\nu_2).
\end{align*}
With similar notations one gets the $n^{\textrm{th}}$ pre-Lie insertion element for an operation $\mu$ of arity $l$:
\begin{align*}
 \mu\triangleleft(\nu_1,\ldots, \nu_n)=\sum_{\{m_1,\ldots,m_{n}\}\subset\{1,\ldots,l\}}^l \mu\circ_{m_1,\ldots,m_{n}}(\nu_1,\ldots,\nu_n)
\end{align*}
where the $m_j$'s are all distinct.  The $n^{\textrm{th}}$ pre-Lie insertion element for an operation $\mu$ of arity $l< n$ is therefore $0$.
\end{proof}

\begin{proof}[Proof of the theorem]
Any operad $\mathcal P$ can be presented as a quotient of the free operad on some generating operations $\mu_1^2,\ldots, \mu_{k_1}^2,\mu_1^3,\ldots, \mu^3_{k_2},\ldots$, where $\mu^i_j$ is an operation of arity $i$. Then for every generating operation $\mu^i_j$ and for any $\nu_1,\ldots,\nu_{i+1}$ operations of $\mathcal P$, the pre-Lie insertion element $\mu_j^i\triangleleft(\nu_1,\ldots,\nu_{i+1})$ vanishes. This provides a relation which is not pre-Lie.  
\end{proof}

\begin{figure}
\begin{tikzpicture}
[level distance=0mm,edge from parent/.style={draw=none},
every node/.style={circle,inner sep=1pt, draw}]
\node{1}[grow'=up];
\end{tikzpicture}
$\ \triangleleft\ \Big($
\begin{tikzpicture}
[level distance=0mm,edge from parent/.style={draw=none},
every node/.style={circle,inner sep=1pt, draw}]
\node{3}[grow'=up]
child {node {2}}
child {};
\end{tikzpicture}
$\Big)=$
\begin{tikzpicture}
[level distance=10mm,
every node/.style={circle,inner sep=1pt, draw}]
\node  {1} [grow'=up]
child {node {2}}
child {node {3}};
\end{tikzpicture}
\caption{ \label{nonunitary}}
\end{figure}
\begin{remark}
In $\mathcal P$, the first vanishing insertion element is given by 
$$\hbox{id}\triangleleft(\hbox{id},\hbox{id})=0.$$
This relation already proves that the pre-Lie algebra $\mathcal P$ cannot be free (and of course $\underline{\mathcal P}$ neither). In the pro-unipotent version $\mathcal P^+$, the simplest insertion relations are given by
$$\mu\triangleleft(\mu,\mu,\mu)=0$$
with $\mu$ any generator of arity two, when such a generator exists in $\mathcal P$.
\end{remark}
\begin{remark}
The relations provided by the vanishing of insertion elements are not the only ones appearing: others come from the associativity axioms. For example  let us consider the free operad on one generator, namely $\mbox{Mag}_2=\mathcal F(\def\svgscale{0.02}
\begingroup%
  \makeatletter%
  \providecommand\color[2][]{%
    \errmessage{(Inkscape) Color is used for the text in Inkscape, but the package 'color.sty' is not loaded}%
    \renewcommand\color[2][]{}%
  }%
  \providecommand\transparent[1]{%
    \errmessage{(Inkscape) Transparency is used (non-zero) for the text in Inkscape, but the package 'transparent.sty' is not loaded}%
    \renewcommand\transparent[1]{}%
  }%
  \providecommand\rotatebox[2]{#2}%
  \ifx\svgwidth\undefined%
    \setlength{\unitlength}{493.40162673bp}%
    \ifx\svgscale\undefined%
      \relax%
    \else%
      \setlength{\unitlength}{\unitlength * \real{\svgscale}}%
    \fi%
  \else%
    \setlength{\unitlength}{\svgwidth}%
  \fi%
  \global\let\svgwidth\undefined%
  \global\let\svgscale\undefined%
  \makeatother%
  \begin{picture}(1,0.34069187)%
    \put(0,0){\includegraphics[width=\unitlength]{arbreA.eps}}%
  \end{picture}%
\endgroup%
)$. Computing the map from $\mbox{Mag}_2$ to $(\mbox{PreLie}(V)/_R,\curvearrowleft)$ in low arities, one gets that $\def\svgscale{0.02}
\begingroup%
  \makeatletter%
  \providecommand\color[2][]{%
    \errmessage{(Inkscape) Color is used for the text in Inkscape, but the package 'color.sty' is not loaded}%
    \renewcommand\color[2][]{}%
  }%
  \providecommand\transparent[1]{%
    \errmessage{(Inkscape) Transparency is used (non-zero) for the text in Inkscape, but the package 'transparent.sty' is not loaded}%
    \renewcommand\transparent[1]{}%
  }%
  \providecommand\rotatebox[2]{#2}%
  \ifx\svgwidth\undefined%
    \setlength{\unitlength}{493.40162673bp}%
    \ifx\svgscale\undefined%
      \relax%
    \else%
      \setlength{\unitlength}{\unitlength * \real{\svgscale}}%
    \fi%
  \else%
    \setlength{\unitlength}{\svgwidth}%
  \fi%
  \global\let\svgwidth\undefined%
  \global\let\svgscale\undefined%
  \makeatother%
  \begin{picture}(1,0.34069187)%
    \put(0,0){\includegraphics[width=\unitlength]{arbreA.eps}}%
  \end{picture}%
\endgroup%
$ will be sent on roots labelled by $\bullet^{\def\svgscale{0.02}
\begingroup%
  \makeatletter%
  \providecommand\color[2][]{%
    \errmessage{(Inkscape) Color is used for the text in Inkscape, but the package 'color.sty' is not loaded}%
    \renewcommand\color[2][]{}%
  }%
  \providecommand\transparent[1]{%
    \errmessage{(Inkscape) Transparency is used (non-zero) for the text in Inkscape, but the package 'transparent.sty' is not loaded}%
    \renewcommand\transparent[1]{}%
  }%
  \providecommand\rotatebox[2]{#2}%
  \ifx\svgwidth\undefined%
    \setlength{\unitlength}{493.40162673bp}%
    \ifx\svgscale\undefined%
      \relax%
    \else%
      \setlength{\unitlength}{\unitlength * \real{\svgscale}}%
    \fi%
  \else%
    \setlength{\unitlength}{\svgwidth}%
  \fi%
  \global\let\svgwidth\undefined%
  \global\let\svgscale\undefined%
  \makeatother%
  \begin{picture}(1,0.34069187)%
    \put(0,0){\includegraphics[width=\unitlength]{arbreA.eps}}%
  \end{picture}%
\endgroup%
}$. Then $V\supset\{\bullet^{\def\svgscale{0.02}
\begingroup%
  \makeatletter%
  \providecommand\color[2][]{%
    \errmessage{(Inkscape) Color is used for the text in Inkscape, but the package 'color.sty' is not loaded}%
    \renewcommand\color[2][]{}%
  }%
  \providecommand\transparent[1]{%
    \errmessage{(Inkscape) Transparency is used (non-zero) for the text in Inkscape, but the package 'transparent.sty' is not loaded}%
    \renewcommand\transparent[1]{}%
  }%
  \providecommand\rotatebox[2]{#2}%
  \ifx\svgwidth\undefined%
    \setlength{\unitlength}{493.40162673bp}%
    \ifx\svgscale\undefined%
      \relax%
    \else%
      \setlength{\unitlength}{\unitlength * \real{\svgscale}}%
    \fi%
  \else%
    \setlength{\unitlength}{\svgwidth}%
  \fi%
  \global\let\svgwidth\undefined%
  \global\let\svgscale\undefined%
  \makeatother%
  \begin{picture}(1,0.34069187)%
    \put(0,0){\includegraphics[width=\unitlength]{arbreA.eps}}%
  \end{picture}%
\endgroup%
},\bullet^{\def\svgscale{0.03}
\begingroup%
  \makeatletter%
  \providecommand\color[2][]{%
    \errmessage{(Inkscape) Color is used for the text in Inkscape, but the package 'color.sty' is not loaded}%
    \renewcommand\color[2][]{}%
  }%
  \providecommand\transparent[1]{%
    \errmessage{(Inkscape) Transparency is used (non-zero) for the text in Inkscape, but the package 'transparent.sty' is not loaded}%
    \renewcommand\transparent[1]{}%
  }%
  \providecommand\rotatebox[2]{#2}%
  \ifx\svgwidth\undefined%
    \setlength{\unitlength}{544.30614828bp}%
    \ifx\svgscale\undefined%
      \relax%
    \else%
      \setlength{\unitlength}{\unitlength * \real{\svgscale}}%
    \fi%
  \else%
    \setlength{\unitlength}{\svgwidth}%
  \fi%
  \global\let\svgwidth\undefined%
  \global\let\svgscale\undefined%
  \makeatother%
  \begin{picture}(1,0.3310869)%
    \put(0,0){\includegraphics[width=\unitlength]{arbreBA.eps}}%
  \end{picture}%
\endgroup%
},\bullet^{\def\svgscale{0.04}
\begingroup%
  \makeatletter%
  \providecommand\color[2][]{%
    \errmessage{(Inkscape) Color is used for the text in Inkscape, but the package 'color.sty' is not loaded}%
    \renewcommand\color[2][]{}%
  }%
  \providecommand\transparent[1]{%
    \errmessage{(Inkscape) Transparency is used (non-zero) for the text in Inkscape, but the package 'transparent.sty' is not loaded}%
    \renewcommand\transparent[1]{}%
  }%
  \providecommand\rotatebox[2]{#2}%
  \ifx\svgwidth\undefined%
    \setlength{\unitlength}{544.30614828bp}%
    \ifx\svgscale\undefined%
      \relax%
    \else%
      \setlength{\unitlength}{\unitlength * \real{\svgscale}}%
    \fi%
  \else%
    \setlength{\unitlength}{\svgwidth}%
  \fi%
  \global\let\svgwidth\undefined%
  \global\let\svgscale\undefined%
  \makeatother%
  \begin{picture}(1,0.3310869)%
    \put(0,0){\includegraphics[width=\unitlength]{arbreCBA.eps}}%
  \end{picture}%
\endgroup%
},\bullet^{\def\svgscale{0.04}
\begingroup%
  \makeatletter%
  \providecommand\color[2][]{%
    \errmessage{(Inkscape) Color is used for the text in Inkscape, but the package 'color.sty' is not loaded}%
    \renewcommand\color[2][]{}%
  }%
  \providecommand\transparent[1]{%
    \errmessage{(Inkscape) Transparency is used (non-zero) for the text in Inkscape, but the package 'transparent.sty' is not loaded}%
    \renewcommand\transparent[1]{}%
  }%
  \providecommand\rotatebox[2]{#2}%
  \ifx\svgwidth\undefined%
    \setlength{\unitlength}{544.30614828bp}%
    \ifx\svgscale\undefined%
      \relax%
    \else%
      \setlength{\unitlength}{\unitlength * \real{\svgscale}}%
    \fi%
  \else%
    \setlength{\unitlength}{\svgwidth}%
  \fi%
  \global\let\svgwidth\undefined%
  \global\let\svgscale\undefined%
  \makeatother%
  \begin{picture}(1,0.3310869)%
    \put(0,0){\includegraphics[width=\unitlength]{arbreDCBA.eps}}%
  \end{picture}%
\endgroup%
},\bullet^{\def\svgscale{0.04}
\begingroup%
  \makeatletter%
  \providecommand\color[2][]{%
    \errmessage{(Inkscape) Color is used for the text in Inkscape, but the package 'color.sty' is not loaded}%
    \renewcommand\color[2][]{}%
  }%
  \providecommand\transparent[1]{%
    \errmessage{(Inkscape) Transparency is used (non-zero) for the text in Inkscape, but the package 'transparent.sty' is not loaded}%
    \renewcommand\transparent[1]{}%
  }%
  \providecommand\rotatebox[2]{#2}%
  \ifx\svgwidth\undefined%
    \setlength{\unitlength}{544.30614828bp}%
    \ifx\svgscale\undefined%
      \relax%
    \else%
      \setlength{\unitlength}{\unitlength * \real{\svgscale}}%
    \fi%
  \else%
    \setlength{\unitlength}{\svgwidth}%
  \fi%
  \global\let\svgwidth\undefined%
  \global\let\svgscale\undefined%
  \makeatother%
  \begin{picture}(1,0.3310869)%
    \put(0,0){\includegraphics[width=\unitlength]{arbreADBA.eps}}%
  \end{picture}%
\endgroup%
} \ldots\}$. The following relation \begin{align*}
6\ \def\svgscale{0.02}
\begingroup%
  \makeatletter%
  \providecommand\color[2][]{%
    \errmessage{(Inkscape) Color is used for the text in Inkscape, but the package 'color.sty' is not loaded}%
    \renewcommand\color[2][]{}%
  }%
  \providecommand\transparent[1]{%
    \errmessage{(Inkscape) Transparency is used (non-zero) for the text in Inkscape, but the package 'transparent.sty' is not loaded}%
    \renewcommand\transparent[1]{}%
  }%
  \providecommand\rotatebox[2]{#2}%
  \ifx\svgwidth\undefined%
    \setlength{\unitlength}{493.40162673bp}%
    \ifx\svgscale\undefined%
      \relax%
    \else%
      \setlength{\unitlength}{\unitlength * \real{\svgscale}}%
    \fi%
  \else%
    \setlength{\unitlength}{\svgwidth}%
  \fi%
  \global\let\svgwidth\undefined%
  \global\let\svgscale\undefined%
  \makeatother%
  \begin{picture}(1,0.34069187)%
    \put(0,0){\includegraphics[width=\unitlength]{arbreA.eps}}%
  \end{picture}%
\endgroup%
\triangleleft\big((\def\svgscale{0.02}
\begingroup%
  \makeatletter%
  \providecommand\color[2][]{%
    \errmessage{(Inkscape) Color is used for the text in Inkscape, but the package 'color.sty' is not loaded}%
    \renewcommand\color[2][]{}%
  }%
  \providecommand\transparent[1]{%
    \errmessage{(Inkscape) Transparency is used (non-zero) for the text in Inkscape, but the package 'transparent.sty' is not loaded}%
    \renewcommand\transparent[1]{}%
  }%
  \providecommand\rotatebox[2]{#2}%
  \ifx\svgwidth\undefined%
    \setlength{\unitlength}{493.40162673bp}%
    \ifx\svgscale\undefined%
      \relax%
    \else%
      \setlength{\unitlength}{\unitlength * \real{\svgscale}}%
    \fi%
  \else%
    \setlength{\unitlength}{\svgwidth}%
  \fi%
  \global\let\svgwidth\undefined%
  \global\let\svgscale\undefined%
  \makeatother%
  \begin{picture}(1,0.34069187)%
    \put(0,0){\includegraphics[width=\unitlength]{arbreA.eps}}%
  \end{picture}%
\endgroup%
\triangleleft(\def\svgscale{0.02}
\begingroup%
  \makeatletter%
  \providecommand\color[2][]{%
    \errmessage{(Inkscape) Color is used for the text in Inkscape, but the package 'color.sty' is not loaded}%
    \renewcommand\color[2][]{}%
  }%
  \providecommand\transparent[1]{%
    \errmessage{(Inkscape) Transparency is used (non-zero) for the text in Inkscape, but the package 'transparent.sty' is not loaded}%
    \renewcommand\transparent[1]{}%
  }%
  \providecommand\rotatebox[2]{#2}%
  \ifx\svgwidth\undefined%
    \setlength{\unitlength}{493.40162673bp}%
    \ifx\svgscale\undefined%
      \relax%
    \else%
      \setlength{\unitlength}{\unitlength * \real{\svgscale}}%
    \fi%
  \else%
    \setlength{\unitlength}{\svgwidth}%
  \fi%
  \global\let\svgwidth\undefined%
  \global\let\svgscale\undefined%
  \makeatother%
  \begin{picture}(1,0.34069187)%
    \put(0,0){\includegraphics[width=\unitlength]{arbreA.eps}}%
  \end{picture}%
\endgroup%
, \def\svgscale{0.02}
\begingroup%
  \makeatletter%
  \providecommand\color[2][]{%
    \errmessage{(Inkscape) Color is used for the text in Inkscape, but the package 'color.sty' is not loaded}%
    \renewcommand\color[2][]{}%
  }%
  \providecommand\transparent[1]{%
    \errmessage{(Inkscape) Transparency is used (non-zero) for the text in Inkscape, but the package 'transparent.sty' is not loaded}%
    \renewcommand\transparent[1]{}%
  }%
  \providecommand\rotatebox[2]{#2}%
  \ifx\svgwidth\undefined%
    \setlength{\unitlength}{493.40162673bp}%
    \ifx\svgscale\undefined%
      \relax%
    \else%
      \setlength{\unitlength}{\unitlength * \real{\svgscale}}%
    \fi%
  \else%
    \setlength{\unitlength}{\svgwidth}%
  \fi%
  \global\let\svgwidth\undefined%
  \global\let\svgscale\undefined%
  \makeatother%
  \begin{picture}(1,0.34069187)%
    \put(0,0){\includegraphics[width=\unitlength]{arbreA.eps}}%
  \end{picture}%
\endgroup%
)),(\def\svgscale{0.02}
\begingroup%
  \makeatletter%
  \providecommand\color[2][]{%
    \errmessage{(Inkscape) Color is used for the text in Inkscape, but the package 'color.sty' is not loaded}%
    \renewcommand\color[2][]{}%
  }%
  \providecommand\transparent[1]{%
    \errmessage{(Inkscape) Transparency is used (non-zero) for the text in Inkscape, but the package 'transparent.sty' is not loaded}%
    \renewcommand\transparent[1]{}%
  }%
  \providecommand\rotatebox[2]{#2}%
  \ifx\svgwidth\undefined%
    \setlength{\unitlength}{493.40162673bp}%
    \ifx\svgscale\undefined%
      \relax%
    \else%
      \setlength{\unitlength}{\unitlength * \real{\svgscale}}%
    \fi%
  \else%
    \setlength{\unitlength}{\svgwidth}%
  \fi%
  \global\let\svgwidth\undefined%
  \global\let\svgscale\undefined%
  \makeatother%
  \begin{picture}(1,0.34069187)%
    \put(0,0){\includegraphics[width=\unitlength]{arbreA.eps}}%
  \end{picture}%
\endgroup%
\triangleleft(\def\svgscale{0.02}
\begingroup%
  \makeatletter%
  \providecommand\color[2][]{%
    \errmessage{(Inkscape) Color is used for the text in Inkscape, but the package 'color.sty' is not loaded}%
    \renewcommand\color[2][]{}%
  }%
  \providecommand\transparent[1]{%
    \errmessage{(Inkscape) Transparency is used (non-zero) for the text in Inkscape, but the package 'transparent.sty' is not loaded}%
    \renewcommand\transparent[1]{}%
  }%
  \providecommand\rotatebox[2]{#2}%
  \ifx\svgwidth\undefined%
    \setlength{\unitlength}{493.40162673bp}%
    \ifx\svgscale\undefined%
      \relax%
    \else%
      \setlength{\unitlength}{\unitlength * \real{\svgscale}}%
    \fi%
  \else%
    \setlength{\unitlength}{\svgwidth}%
  \fi%
  \global\let\svgwidth\undefined%
  \global\let\svgscale\undefined%
  \makeatother%
  \begin{picture}(1,0.34069187)%
    \put(0,0){\includegraphics[width=\unitlength]{arbreA.eps}}%
  \end{picture}%
\endgroup%
,\def\svgscale{0.02}
\begingroup%
  \makeatletter%
  \providecommand\color[2][]{%
    \errmessage{(Inkscape) Color is used for the text in Inkscape, but the package 'color.sty' is not loaded}%
    \renewcommand\color[2][]{}%
  }%
  \providecommand\transparent[1]{%
    \errmessage{(Inkscape) Transparency is used (non-zero) for the text in Inkscape, but the package 'transparent.sty' is not loaded}%
    \renewcommand\transparent[1]{}%
  }%
  \providecommand\rotatebox[2]{#2}%
  \ifx\svgwidth\undefined%
    \setlength{\unitlength}{493.40162673bp}%
    \ifx\svgscale\undefined%
      \relax%
    \else%
      \setlength{\unitlength}{\unitlength * \real{\svgscale}}%
    \fi%
  \else%
    \setlength{\unitlength}{\svgwidth}%
  \fi%
  \global\let\svgwidth\undefined%
  \global\let\svgscale\undefined%
  \makeatother%
  \begin{picture}(1,0.34069187)%
    \put(0,0){\includegraphics[width=\unitlength]{arbreA.eps}}%
  \end{picture}%
\endgroup%
)) \big)=(\def\svgscale{0.02}
\begingroup%
  \makeatletter%
  \providecommand\color[2][]{%
    \errmessage{(Inkscape) Color is used for the text in Inkscape, but the package 'color.sty' is not loaded}%
    \renewcommand\color[2][]{}%
  }%
  \providecommand\transparent[1]{%
    \errmessage{(Inkscape) Transparency is used (non-zero) for the text in Inkscape, but the package 'transparent.sty' is not loaded}%
    \renewcommand\transparent[1]{}%
  }%
  \providecommand\rotatebox[2]{#2}%
  \ifx\svgwidth\undefined%
    \setlength{\unitlength}{493.40162673bp}%
    \ifx\svgscale\undefined%
      \relax%
    \else%
      \setlength{\unitlength}{\unitlength * \real{\svgscale}}%
    \fi%
  \else%
    \setlength{\unitlength}{\svgwidth}%
  \fi%
  \global\let\svgwidth\undefined%
  \global\let\svgscale\undefined%
  \makeatother%
  \begin{picture}(1,0.34069187)%
    \put(0,0){\includegraphics[width=\unitlength]{arbreA.eps}}%
  \end{picture}%
\endgroup%
\triangleleft(\def\svgscale{0.02}
\begingroup%
  \makeatletter%
  \providecommand\color[2][]{%
    \errmessage{(Inkscape) Color is used for the text in Inkscape, but the package 'color.sty' is not loaded}%
    \renewcommand\color[2][]{}%
  }%
  \providecommand\transparent[1]{%
    \errmessage{(Inkscape) Transparency is used (non-zero) for the text in Inkscape, but the package 'transparent.sty' is not loaded}%
    \renewcommand\transparent[1]{}%
  }%
  \providecommand\rotatebox[2]{#2}%
  \ifx\svgwidth\undefined%
    \setlength{\unitlength}{493.40162673bp}%
    \ifx\svgscale\undefined%
      \relax%
    \else%
      \setlength{\unitlength}{\unitlength * \real{\svgscale}}%
    \fi%
  \else%
    \setlength{\unitlength}{\svgwidth}%
  \fi%
  \global\let\svgwidth\undefined%
  \global\let\svgscale\undefined%
  \makeatother%
  \begin{picture}(1,0.34069187)%
    \put(0,0){\includegraphics[width=\unitlength]{arbreA.eps}}%
  \end{picture}%
\endgroup%
,\def\svgscale{0.02}
\begingroup%
  \makeatletter%
  \providecommand\color[2][]{%
    \errmessage{(Inkscape) Color is used for the text in Inkscape, but the package 'color.sty' is not loaded}%
    \renewcommand\color[2][]{}%
  }%
  \providecommand\transparent[1]{%
    \errmessage{(Inkscape) Transparency is used (non-zero) for the text in Inkscape, but the package 'transparent.sty' is not loaded}%
    \renewcommand\transparent[1]{}%
  }%
  \providecommand\rotatebox[2]{#2}%
  \ifx\svgwidth\undefined%
    \setlength{\unitlength}{493.40162673bp}%
    \ifx\svgscale\undefined%
      \relax%
    \else%
      \setlength{\unitlength}{\unitlength * \real{\svgscale}}%
    \fi%
  \else%
    \setlength{\unitlength}{\svgwidth}%
  \fi%
  \global\let\svgwidth\undefined%
  \global\let\svgscale\undefined%
  \makeatother%
  \begin{picture}(1,0.34069187)%
    \put(0,0){\includegraphics[width=\unitlength]{arbreA.eps}}%
  \end{picture}%
\endgroup%
))\triangleleft(\def\svgscale{0.02}
\begingroup%
  \makeatletter%
  \providecommand\color[2][]{%
    \errmessage{(Inkscape) Color is used for the text in Inkscape, but the package 'color.sty' is not loaded}%
    \renewcommand\color[2][]{}%
  }%
  \providecommand\transparent[1]{%
    \errmessage{(Inkscape) Transparency is used (non-zero) for the text in Inkscape, but the package 'transparent.sty' is not loaded}%
    \renewcommand\transparent[1]{}%
  }%
  \providecommand\rotatebox[2]{#2}%
  \ifx\svgwidth\undefined%
    \setlength{\unitlength}{493.40162673bp}%
    \ifx\svgscale\undefined%
      \relax%
    \else%
      \setlength{\unitlength}{\unitlength * \real{\svgscale}}%
    \fi%
  \else%
    \setlength{\unitlength}{\svgwidth}%
  \fi%
  \global\let\svgwidth\undefined%
  \global\let\svgscale\undefined%
  \makeatother%
  \begin{picture}(1,0.34069187)%
    \put(0,0){\includegraphics[width=\unitlength]{arbreA.eps}}%
  \end{picture}%
\endgroup%
,\def\svgscale{0.02}
\begingroup%
  \makeatletter%
  \providecommand\color[2][]{%
    \errmessage{(Inkscape) Color is used for the text in Inkscape, but the package 'color.sty' is not loaded}%
    \renewcommand\color[2][]{}%
  }%
  \providecommand\transparent[1]{%
    \errmessage{(Inkscape) Transparency is used (non-zero) for the text in Inkscape, but the package 'transparent.sty' is not loaded}%
    \renewcommand\transparent[1]{}%
  }%
  \providecommand\rotatebox[2]{#2}%
  \ifx\svgwidth\undefined%
    \setlength{\unitlength}{493.40162673bp}%
    \ifx\svgscale\undefined%
      \relax%
    \else%
      \setlength{\unitlength}{\unitlength * \real{\svgscale}}%
    \fi%
  \else%
    \setlength{\unitlength}{\svgwidth}%
  \fi%
  \global\let\svgwidth\undefined%
  \global\let\svgscale\undefined%
  \makeatother%
  \begin{picture}(1,0.34069187)%
    \put(0,0){\includegraphics[width=\unitlength]{arbreA.eps}}%
  \end{picture}%
\endgroup%
,\def\svgscale{0.02}
\begingroup%
  \makeatletter%
  \providecommand\color[2][]{%
    \errmessage{(Inkscape) Color is used for the text in Inkscape, but the package 'color.sty' is not loaded}%
    \renewcommand\color[2][]{}%
  }%
  \providecommand\transparent[1]{%
    \errmessage{(Inkscape) Transparency is used (non-zero) for the text in Inkscape, but the package 'transparent.sty' is not loaded}%
    \renewcommand\transparent[1]{}%
  }%
  \providecommand\rotatebox[2]{#2}%
  \ifx\svgwidth\undefined%
    \setlength{\unitlength}{493.40162673bp}%
    \ifx\svgscale\undefined%
      \relax%
    \else%
      \setlength{\unitlength}{\unitlength * \real{\svgscale}}%
    \fi%
  \else%
    \setlength{\unitlength}{\svgwidth}%
  \fi%
  \global\let\svgwidth\undefined%
  \global\let\svgscale\undefined%
  \makeatother%
  \begin{picture}(1,0.34069187)%
    \put(0,0){\includegraphics[width=\unitlength]{arbreA.eps}}%
  \end{picture}%
\endgroup%
,\def\svgscale{0.02}
\begingroup%
  \makeatletter%
  \providecommand\color[2][]{%
    \errmessage{(Inkscape) Color is used for the text in Inkscape, but the package 'color.sty' is not loaded}%
    \renewcommand\color[2][]{}%
  }%
  \providecommand\transparent[1]{%
    \errmessage{(Inkscape) Transparency is used (non-zero) for the text in Inkscape, but the package 'transparent.sty' is not loaded}%
    \renewcommand\transparent[1]{}%
  }%
  \providecommand\rotatebox[2]{#2}%
  \ifx\svgwidth\undefined%
    \setlength{\unitlength}{493.40162673bp}%
    \ifx\svgscale\undefined%
      \relax%
    \else%
      \setlength{\unitlength}{\unitlength * \real{\svgscale}}%
    \fi%
  \else%
    \setlength{\unitlength}{\svgwidth}%
  \fi%
  \global\let\svgwidth\undefined%
  \global\let\svgscale\undefined%
  \makeatother%
  \begin{picture}(1,0.34069187)%
    \put(0,0){\includegraphics[width=\unitlength]{arbreA.eps}}%
  \end{picture}%
\endgroup%
)
\end{align*}
comes from the associativity axioms. But it does not come from the vanishing induced by the insertion elements and is not a pre-Lie relation.
\end{remark}
Finally, as the operads pre-Lie and NAP share the same underlying $\mathbb S$-module, on the free Pre-Lie algebra over a vector space, the pre-Lie product of two elements $x\triangleleft y$  can be viewed as a polynomial of elements obtained with the NAP-product $\LHD$. To have a better understanding of the operad $\mbox{Mag}_2$ as a pre-Lie algebra,  it is natural to try to  perform an analogue of Grobner basis algorithm by defining leading power product $LPP(x)$ for every $x\in \mbox{Mag}_2$ which is $NAP$-multiplicative with respect to the pre-Lie product of elements :  $\forall x,y \in \mbox{Mag}_2$ the leading power product of $x\triangleleft y$ satisfies    $LPP(x\triangleleft y)=LPP(x)\LHD LPP(y)$  where $\LHD$ is a NAP product on $\mbox{Mag}_2$ see for example \cite{Dzhumadildaev-Lofwall}.

It can be shown by direct inspection, that there exists  no such NAP algebra structure on $\mbox{Mag}_2$ endowed with the above pre-Lie structure.  This fails with trees with $5$ leaves, taking into account the constraints from trees with less leaves, the NAP relations and the pre-Lie relations.


\begin{thebibliography}{Cha02}

\bibitem[Cha02]{Chapoton}
F.~Chapoton.
\newblock Rooted trees and exponential-like series.
\newblock arXiv:math/0209104, 2002.

\bibitem[CL01]{Chapoton-Livernet}
F.~Chapoton and M.~Livernet.
\newblock Pre-{L}ie algebras and the rooted trees operad.
\newblock {\em Internat. Math. Res. Notices}, (8):395--408, 2001.

\bibitem[DL02]{Dzhumadildaev-Lofwall}
A.~Dhzumadil'daev and C.~L\"ofwall.
\newblock Trees, free right-symmetric algebras, free novikov algebras and
  identities.
\newblock {\em Homology, Homotopy and Applications}, 4(2):165--190, 2002.

\bibitem[FM15]{FM15}
A.~Frabetti and D.~Manchon.
\newblock Five interpretations of the {F}a\`a di {B}runo formula.
\newblock In {F}.~{F}auvet K.~{E}brahimi {F}ard, editor, {\em Fa\`a di {B}runo
  {H}opf algebras, {D}yson-{S}chwinger equations, and {L}ie-{B}utcher series},
  volume~21 of {\em {IRMA} {L}ectures in {M}athematics and {T}heoretical
  {P}hysics}, pages 91--147. {E}urop. {M}ath. {S}oc., 2015.

\bibitem[KM01]{KM01}
M.~Kapranov and Y.~Manin.
\newblock Modules and {M}orita theorem for operads.
\newblock {\em American {J}. {M}ath.}, 123(5):811--838, 2001.

\bibitem[LV12]{Loday-Vallette}
J.-L. Loday and B.~Vallette.
\newblock {\em Algebraic operads}, volume 346 of {\em Grundlehren der
  Mathematischen Wissenschaften [Fundamental Principles of Mathematical
  Sciences]}.
\newblock Springer, Heidelberg, 2012.

\bibitem[MS11]{Manchon-Saidi}
D.~Manchon and A.~Sa{\"{\i}}di.
\newblock Lois pr\'e-{L}ie en interaction.
\newblock {\em Comm. Algebra}, 39(10):3662--3680, 2011.

\end{thebibliography}
\end{document}